\documentclass[12pt, reqno]{amsart}

\title[Bridges of Markov counting processes]{Bridges of Markov counting processes. Reciprocal classes and duality formulas}

\author{Giovanni Conforti}
\address{Institut f\"ur Mathematik der Universit\"at Potsdam. Am Neuen Palais 10. 14469 Potsdam, Germany}
\email{giovanniconfort@gmail.com}

\author{Christian L\'eonard }
\address{Modal-X. Universit\'e Paris Ouest. B\^at.\! G, 200 av. de la R\'epublique. 92001 Nanterre, France}
\email{christian.leonard@u-paris10.fr}

\author{R\"udiger Murr}
\address{Institut f\"ur Mathematik der Universit\"at Potsdam. Am neuen Palais 10. 14469 Potsdam, Germany}
\email{rudiger.murr@gmail.com}

\author{Sylvie R\oe lly}
\address{Institut f\"ur Mathematik der Universit\"at Potsdam. Am Neuen Palais 10. 14469 Potsdam, Germany}
\email{roelly@math.uni-potsdam.de}

\thanks{The authors are grateful to the French-German University (UFA/DFH) who supported this work through the CDFA 01-06. 
Second author was partially supported by the ANR project GeMeCoD. ANR 2011 BS01 007 01}
 
 \keywords{}
 \subjclass[2010]{}

\usepackage{amssymb, amsmath, amsfonts, latexsym, enumerate}
\usepackage[english]{babel}
\usepackage[LGR,T1]{fontenc}
\usepackage[utf8]{inputenc}
\RequirePackage[colorlinks,linkcolor=blue,citecolor=blue,urlcolor=blue]{hyperref}

\usepackage[usenames,dvipsnames]{pstricks} 

\newtheorem{theorem}[equation]{Theorem}
\newtheorem{lemma}[equation]{Lemma}
\newtheorem{proposition}[equation]{Proposition}
\newtheorem{corollary}[equation]{Corollary}

\newtheorem{definition}[equation]{Definition}

\theoremstyle{remark}
\newtheorem{remark}[equation]{Remark}

\newtheorem{example}[equation]{Example}

\numberwithin{equation}{section}

\newcommand{\R}{\mathbb{R}}
\newcommand{\Z}{\mathbb{Z}}
\newcommand{\N}{\mathbb{N}}
\newcommand{\Rec}{\mathfrak{R}}

\newcommand{\PO}{\mathcal{P}(\Omega)}

\newcommand{\Poi}{\mathbf{P}}
\newcommand{\Pl}{P_\ell}
\newcommand{\Rl}{R_\ell}
\newcommand{\El}{E_{\ell}}

\newcommand{\EE}{ \mathbf{E}}

\newcommand{\1}{\mathbf{1}}

\newcommand{\e}{\varepsilon}
\newcommand{\I}{[0,1]}

\newcommand\OO{\Omega}

\newcommand{\mC}{\mathcal{C}}

\newcommand{\DD}{\mathcal{D}}

\newcommand{\uu}{\dot u}
\newcommand{\UU}{ \mathcal{U}}

\begin{document}

\begin{abstract} 
Processes having the same bridges are said to belong to the same reciprocal class. In this article we analyze reciprocal classes of Markov counting processes by identifying their reciprocal invariants and we characterize them  as the set of counting processes satisfying some duality formula. 
\end{abstract}

\subjclass[2010]{60H07; 60G51; 60G57}
\keywords{Markov counting process; duality formula; reciprocal class}

\maketitle
\tableofcontents

\section*{Introduction}
The theory of reciprocal processes is rooted in an  early work of Schrödinger \cite{Sch31} on constrained Brownian particles.
It was   developed in the framework of diffusion processes by Jamison, Krener, Thieullen and Zambrini among other contributors, see  \cite{Roe13,LRZ12} for  recent  reviews on this topic. 
After Jamison \cite{Jam74}, it was clear that a central notion of the theory was that of \emph{reciprocal class}  of a given  stochastic process, which is the set of all processes that share their bridges with  this process.

The  aim of the present article  is to investigate  the reciprocal class of the easiest possible processes  with jumps, namely the counting processes. Although based on simple processes, it appears that this reciprocal  structure  is  interesting.

These simple processes with jumps, which we call \emph{nice Markov counting} (NMC, for short) processes and include the standard Poisson process, are introduced in the first section together with their reciprocal classes. It is proved at  Theorem \ref{thm:clarkpoisson}  that two NMC processes belong to the same reciprocal class, i.e.\ have the same bridges,  if and only if  some specific time-space functions derived from their intensities of jumps coincide. This common function is called the   {\it reciprocal invariant} of the reciprocal class.
 We derive at  Lemma \ref{prop:ibpfunit} a duality relation between some stochastic integral and a derivative operator  built on variations of the instants of  jump, which holds  for any  NMC process.
It will lead us to our main result which states at Theorem \ref{thm:derivcharacrec} a characterization of the reciprocal class of a given NMC process by means of a duality formula involving both the reciprocal invariant of the class and small variations of the times of jumps. 
  
This article is partly based on the third author's PhD thesis \cite{Murr12}. Our results are analogous to those of Thieullen and the last author \cite{RT02, RT05} that were obtained in the framework of Brownian diffusion processes. However, the proofs of the present paper  differ significantly from those of these earlier works. 
Recently, similar results  have been obtained in \cite{CDPR14}  for  compound Poisson processes, using variations based on space perturbations, rather than time perturbations as in the present article.

\section{Counting processes and their reciprocal classes}
 
The basic object of this paper, 
called {\it  nice Markov counting}  process (NMC process, for short) is described, the reciprocal class of an NMC process is defined  and its relation with some 
$h$-transforms is made precise. Theorem \ref{thm:clarkpoisson} states a characterization of the reciprocal class of an NMC process in terms of its reciprocal invariant.

\subsection*{Framework, definitions and notation} 
\label{sec:framework}

The sample  path space $ \Omega$ of the counting  processes consists of all c\`adl\`ag step functions with finitely many  jumps with amplitude $+1$ and an initial value in $\Z$. 
Any path $ \omega\in\OO$ is described by the collection $(x;t_1,\dots,t_n)$ of its initial position $x\in\Z$ and its $n= \omega_1- \omega_0$ instants of jumps $0<t_1<\cdots<t_n<1.$ It is practical to set $t_i=1,$ for all $i>n$ and identify $ \omega$ with $[x;(t_i)_{ i\ge 1}].$ We denote $X_0( \omega):=x$ and $T_i( \omega):=t_i$ the $i$-th instant of jump of $ \omega.$

The canonical counting process is denoted by $X=(X_t) _{ 0\le t\le 1}$  and 
$\mathcal{P}(\mathbf{E})$ denotes the space of all probability measures on a measurable space $\mathbf{E}$. For any $Q \in \PO$
and $t \in [0,1]$, the marginal law of $Q$ at time $t$ is denoted by 
$	
Q_t := Q \circ X_t^{-1} \, \in \mathcal{P}(\Z) 
$
and
$
Q_{01}:= Q \circ (X_0, X_1)^{-1} \, \in \mathcal{P}(\Z^2)
$	
is the endpoint marginal law of $Q$.

Any $Q\in\PO$ admits an   increasing process
denoted by $A:[0,1]\times \Omega \rightarrow \R_+$ such that $Q(A(0)= 0)=1$ and 
$	t \mapsto X_t - X_0 - A(t)
$  is a local $Q$-martingale,
(see Jacod \cite[Thm.\ 2.1]{Jac75}, for instance) which characterizes the dynamics of $Q$.
When the compensator is absolutely continuous, 
we call its derivative  the {\it intensity} of $Q$.

Let us introduce the set of the reference processes of this paper.

\begin{definition} \label{def:regularmarkov}
Let $\ell:\I\times\Z\to(0,\infty)$ be  a positive and  upper bounded function: 
$
0<\ell(t,z) \le \bar\lambda ,
$ 
for some $\bar\lambda >0$ and for all $(t,z)\in \I\times\Z,$
such that for each  $z\in \Z,$  $t \mapsto \ell(t,z)$ is in $\mC^1(\I).$
\\
Such a function $\ell$ is called a \emph{nice Markov counting} intensity, an  {NMC} intensity for short. 
A counting process with an NMC intensity  is called  an  {NMC} process. 
\end{definition}
Each NMC process belongs to $\PO$
and the strict  positivity  of the intensity implies that two NMC are absolutely continuous with respect to each other.

Let $\ell$ be an NMC intensity. For any $x\in \Z,$ we define $R _{ \ell}^x \in\PO$ as the law of the NMC process starting from $x$ with intensity $\ell$ and we denote 
for any probability measure $ \mu\in \mathcal{P}(\Z)$, 
$
P ^{ \mu}_\ell(\cdot):=\int _{ \Z}R^x_\ell(\cdot)\, \mu(dx)\in\PO
$
the law of the counting process with intensity $\ell$ and initial law $ \mu.$ 
When the initial law $ \mu$ is not relevant, we drop the corresponding superscript and simply write $\Pl$.
Let us introduce the kernel
\begin{equation}\label{eq-kern}
 \big(R^{ xy}_{ \ell}:=R^x _{ \ell}(\cdot\mid X_1=y); x\le y\in\Z\big) 
\end{equation}
which is a family in $\PO$ indexed by the endpoints $(x,y).$
Clearly, $\Rl ^{ xy}$ is a regular version of the $xy$-bridge of the unbounded  $ \sigma$-finite measure $R_\ell:=\sum _{ x\in\Z}R^x _{ \ell}$ which  is defined everywhere, i.e.\ for all $x\le y\in\Z.$
Of course, $\Pl ^{ xy}=\Rl^{ xy}$ and $\Pl ^{ \mu}(\cdot\mid X_0=x)=\Rl^x$. The kernel \eqref{eq-kern} is the basic dynamical object that will be used later on. It  is uniquely determined by the intensity $\ell.$

\begin{example}\label{expl-Poi}
A fundamental subclass of $\PO$ is the set of Poisson processes 
characterized by two parameters: a constant intensity, say $\alpha>0$, and the distribution $\mu \in \mathcal{P}(\Z)$ of the random initial value. We denote it by $\Poi_{\alpha}^\mu$. 
If  $\alpha=1$ we  drop the index and simply write   $\Poi^\mu$ or $\Poi$. The bridges $R _{ \alpha=1}^{ xy}$ are  denoted by $ \mathbf{R}^{ xy}.$
\end{example}

\begin{proposition}[Girsanov formula] \label{lem:girsanovunit}
Let $\Pl, P_k \in \PO$ be two NMC processes sharing the same initial distribution. Then they are absolutely continuous with respect to each other 
and 
\begin{equation}\label{eq:girsanovklunitjump}
\frac{d\Pl}{dP_k} = \exp\Big(-\int_{[0,1]} (\ell(s,X_{s^-})-k(s,X_{s^-})) ds \Big)
\prod_{i:T_i < 1} \frac{\ell(T_i,X_{T_i^-})}{k(T_i,X_{T_i^-})},
\end{equation}
In particular, the density of $\Pl$ with respect to  a standard Poisson process with the same initial distribution is \begin{equation}\label{eq:girsanovunitjump}
G_{\ell}:= \frac{d\Pl ^{ \mu}}{d\Poi^{ \mu}} = \exp\Big(-\int_{[0,1]} \bigg(\ell(s,X_{s^-})-1\bigg) ds \Big)\prod_{i:T_i<1} \ell(T_i,X_{T_i^-}).
\end{equation}
\end{proposition}
Since $X _{ 1^-}=X_1, \Pl$-a.s., ``$i:T_i<1$'' indicates that only the effective jumps are taken into account.\\
Remark that  $G_{\ell}>0,$ $\Poi^\mu$-a.s.
\begin{proof}
It follows from  the boundedness of the intensity and \cite[Thm.\ 5.1]{Jac75}.
\end{proof}

\subsection*{Reciprocal classes of Markov counting processes} 
\label{sec:defclas} 

Let us define  the reciprocal class of  an NMC intensity. 
\begin{definition}\label{def-clasrec} 
Let $\ell $ be an NMC intensity. The  \emph{reciprocal class} $\Rec (\ell)$ associated to the intensity $\ell$  is the set of all probability measures on 
$\OO$ obtained as mixtures of the bridges $R_{\ell}^{xy}$, see \eqref{eq-kern}, that is 
 \begin{equation}\label{eq-12}
 \Rec (\ell):= \big\{Q \in \PO:   Q(\cdot)=\int_{\Z^2} R _{ \ell}^{xy}(\cdot) \,Q_{01}(dxdy)\big\}.
\end{equation}
\end{definition}

The integral  in \eqref{eq-12} makes sense since  $Q_{01}$  is concentrated on  $ \left\{x\le y\in\Z\right\}$.

Any  element $Q$ of $\Rec( \ell)$ possesses the following interesting time symmetry property, called {\it reciprocal  property}:
For any  $ 0\le s\le u\le 1$ and any  $X _{ [s,u]}$-measurable $A,$
\begin{equation*}
Q(A\mid X _{ [0,s]}, X _{ [u,1]})=Q(A\mid X_s,X_u).
\end{equation*}
Bernstein put forward this property  in \cite{Bern32}.
The reciprocal property is weaker than the Markov one. We will encounter in the rest of the paper path measures which are not Markov but are reciprocal.

A better understanding of the dynamics of the processes in  $\Rec (\ell)$ will follow from next result.
\begin{proposition}\label{prop:Q=hP} 
The path measure $Q \in \PO $ belongs to $\Rec(\ell)$ if and only if there exists a nonnegative measurable function $h:\Z^2  \rightarrow \R_+$ 
 such that
\begin{equation} \label{htransform}
Q = h(X_0,X_1) \, R _{ \ell}.
\end{equation}
Recall that $\Rl:=\sum _{ x\in\Z} \Rl ^x.$
\end{proposition}

\begin{proof} See \cite{LRZ12}.
\end{proof}

\begin{remark}
If the density $h=dQ/d\Rl$ is a function of the final state $X_1$ only, one identifies $Q$ as an $h$-transform of  $\Rl$ in the sense of Doob \cite{Doob57}.
\end{remark}

\begin{example} 
\label{ex:Poissonbridge}
 Let $\Poi$ be the standard Poisson process with unit intensity and  initial condition 0. Recall the notation of Example \ref{expl-Poi}.
\begin{enumerate}[(a)]
\item
Any Poisson process $\Poi_{\alpha}$ belongs to the reciprocal class $\Rec (1)$ associated with the  unit intensity $\ell=1$, whatever is its intensity $\alpha$.  Indeed, for each $x \in \Z$, the identity
$
\Poi_\alpha^x = e^{1-\alpha} \, \alpha^{X_1} \, \Poi^{x} $ implies that for any integer $x\le y,$ $\Poi_\alpha^{xy} = \mathbf{R}^{xy} .
$
\item 
 Let $\mu$ be any probability measure on $\N$. 
 The unique probability measure $Q$ in   $\Rec (1)$
 with  endpoint  distribution $Q_{01}=\delta_{\{0\}}\otimes \mu$ is given by 
  $$
  Q= e \, X_1!\, \mu(X_1) \, \Poi.
$$
In  next section, we will obtain  an alternate proof of this former assertion  using  reciprocal invariants.
\end{enumerate}
\end{example}

\subsection*{Reciprocal invariant}\label{sec:compnice} 

We show that the reciprocal class of an NMC intensity  is characterized by a function which is called its \emph{reciprocal invariant}. This new result extends previous studies  by Clark in  the framework of Brownian diffusions. It was proved  in \cite[Thm.\,1]{Cl91}  that reciprocal  classes  of Brownian diffusions are characterized by some  {reciprocal invariants}
that are functions of the drift. Rather than adapting Clark's proof, we  follow the  strategy of Zambrini and the second author \cite{LZ14} in the setting of diffusion processes.

 Let us  describe how the intensity of a counting process behaves through a change of measure involving  the final state only.
 
\begin{proposition}[$h$-transform] \label{lem:h-trafounit}
Let $\Pl$ be the law of an NMC process and $h:\Z\rightarrow (0, \infty)$ be any positive function such that $\El \,h(X_1) = 1$.
Then the process on $\Omega$ whose law is given by  
$$
Q: = h(X_1) \, \Pl
$$
is Markov and its intensity $k:[0,1] \times \Z \rightarrow [0, \infty)$ satisfies:
\begin{equation} \label{eq:h-intensity}
k(t,X_{t^-}) =  \frac{h(t,X_{t^-}+1)}{h(t,X_{t^-})} \, \ell(t,X_{t^-}) \quad dt\otimes Q\mbox{-}a.s.,
\end{equation}
where $h(t,z): = \El (h(X_1)|X_t = z)$.
\end{proposition}
\begin{proof}
This is a special easy case of a Doob $h$-transform, \cite{Doob57}.
First note that $h(t,z)$ is time-differentiable (a standard semigroup argument) and space-time harmonic, that is 
\begin{equation} \label{eq:harmonic}
\partial_t h(t,z) + \ell(t,z)\left[h(t,z+1) - h(t,z) \right] = 0, \quad \Pl \circ X_t^{-1} \mbox{-a.e.},\mbox{ for all }t\in \I.
\end{equation}
Denote  $\psi (t,z):= \log h(t,z)$ which is $\Pl \circ X_t^{-1}$-a.e.\ well defined. It satisfies
$$
\partial_t \psi(t,z) = \ell (t,z) \big(1 - e^{\psi(t,z+1) - \psi(t,z)} \big) .
$$
By It\^o formula, for any $0\le t\le1,$ 
\begin{eqnarray*}
\psi(t,X_t) &=  &  \int_{]0,t]}\partial_t \psi(s,X_{s^-})\, dr + \sum_{i:T_i \le t }\Big(\psi(T_i,X_{T_i}) - \psi(T_i,X_{T_i^-})\Big) \\
& = & \int_{]0,t]}\left(1 - \frac{h(s,X_{s^-}+1)}{h(s,X_{s^-})} \right)\ell(s,X_{s^-})ds  +
\sum_{i:T_i \le t }\log \frac{h(T_i,X_{T_i^-}+1)}{h(T_i,X_{T_i^-})}.
\end{eqnarray*}
Since $ h(X_1)= e^{\psi(1,X_1)}$, this implies that
$$
h(X_1) = 
\exp\left(-\int_{(0,1]}\left(\frac{h(s,X_{s^-}+1)}{h(s, X_{s^-})}-1\right)\ell(s,X_{s^-})ds \right)\prod_{i:T_i<1} \frac{h(T_i,X_{T_i^-}+1)}{h(T_i,X_{T_i^-})}.
$$
 and
\begin{eqnarray*}
h(X_1)G_{\ell} = \exp\left(-\int_{(0,1]}\left(\ell(s,X_{s^-})\frac{h(s,X_{s^-}+1)}{h(s,X_{s^-})}-1\right)ds \right)
\prod_{i:T_i<1} \ell(T_i,X_{T_i^-})\frac{h(T_i,X_{T_i}+1)}{h(T_i,X_{T_i^-})}.
\end{eqnarray*}
With \eqref{eq:girsanovunitjump} we see that $Q$ admits the intensity $k$ as defined in \eqref{eq:h-intensity}.
\end{proof}

\begin{example} 
Take $x\le y\in \Z$. The bridge $\Poi^{xy}_\alpha$ of the Poisson process with intensity $\alpha$ starting from $x $ is given by 
$$
\Poi^{xy}_\alpha = \frac{\1_{\{X_1=y\}}}{\Poi^x_\alpha(X_1 = y)} \, \Poi^x_\alpha = e^\alpha \frac{(y-x)!}{\alpha^{y-x}} \, \1_{\{X_1=y\}}  \, \Poi^x.
$$
Since 
\begin{eqnarray*} 
\Poi^x_\alpha(X_1=y \mid X_t = z) = e^{-\alpha(1-t)} \frac{(\alpha(1-t))^{y-z}}{(y-z)!},
\end{eqnarray*}
this implies that the intensity $k$ of the bridge $\Poi^{xy}_\alpha$ is given for $t \in [0,1)$ by
$$
k(t,z) =  \frac{y-z}{1-t} \, \1_{\{y>z\}},
$$
whose interpretation is as follows. If the process is located at $z$ at time $t$, it will undergo a kind of mean velocity  $(y-z)/(1-t)$ to reach 
the state $y$ during the remaining duration $1-t$. Note the  similarity with the drift of the Brownian bridge. The intensity explodes at time $1$ and it does not depend on $\alpha$, as already mentioned in Example \ref{ex:Poissonbridge}-(a).
\end{example}

We give at Theorem \ref{thm:clarkpoisson} below a necessary and sufficient condition for the equality of the reciprocal classes associated with  
two distinct NMC intensities.
 
\begin{definition}[Reciprocal invariant]
For any NMC intensity  $\ell$, we define the map
\begin{equation} \label{eq:invariantunitjump}
\Xi_{\ell}(t,z) := \partial_t \log \ell(t,z) + \ell(t,z+1)-\ell(t,z)
\end{equation}
and  call it the  \emph{reciprocal invariant} of the class $\Rec({\ell})$.
\end{definition}
This terminology is justified by the  following result.

\begin{theorem}\label{thm:clarkpoisson}
Let $\ell$ and $k$ be two  NMC intensities.
Then, 
$
\Rec (\ell) = \Rec({k}) $ if and only if \ $ \Xi_{\ell} = \Xi_{k}.
$
\end{theorem}
\begin{proof}
Note that for any $x\in \Z$, the measures $P_k^x$ and $P_\ell^x$ are  absolutely continuous with respect to each other. 
Assume that $P_k \in \Rec({\ell})$ and fix $x\in \Z$. By Proposition \ref{prop:Q=hP},
there exists a positive function $h$ such that $P^x_{k} = h(X_1)P_{\ell}^x$.
Therefore $k$ and $\ell$  are related by (\ref{eq:h-intensity}):
$$
k(t,z)=e^{\psi(t,z+1) - \psi(t,z)} \ell(t,z),
$$
where $\psi (t,z):= \log h(t,z)$.
This leads to
\begin{eqnarray*}
 \partial_t (\log k(t,z) - \log \ell(t,z)) &=&   \partial_t (\psi(t,z+1) - \psi(t,z))\\
 &=& \ell (t,z+1) \Big(1 - e^{\psi(t,z+2) - \psi(t,z+1)} \Big) - \ell (t,z) \Big(1 - e^{\psi(t,z+1) - \psi(t,z)} \Big)\\
 &=&  \Big(\ell (t,z+1)- \ell (t,z) \Big)- \Big(k(t,z+1)- k (t,z)\Big)
\end{eqnarray*}
which implies the equality of  $\Xi_k$ and $\Xi_{\ell}$.\\
Let us prove the converse statement. 
Suppose that $\Xi_k = \Xi_{\ell}$. This implies that there exists a regular function $c$ which is only space dependent such that 
$$
\log \frac{k}{\ell}(t,z) = \phi(t,z+1) -\phi(t,z) + c(z)  \textrm{ where } \phi(t,z):= \int_0^t [\ell(s,z) - k (s,z)]\, ds.
$$
By Equation \eqref{eq:girsanovklunitjump}, for any $x\in \Z$, we have
\begin{equation*}
\frac{dP_k^x}{dP_\ell^x} = 
\exp\Big(\int_{(0,1]} [ \ell(t,X_{t^-})-k(t,X_{t^-})] \ dt \Big)\prod_{i:T_i<1} \frac{k}{\ell}(T_i,X_{T_i^-}),
\end{equation*}
where only the random instants $T_i< 1$ of  jumps are taken into account in the last product.
Therefore,
\begin{eqnarray*}
 \frac{dP_k^x}{dP_\ell^x} &=&  \exp\Big(\int_{(0,1]} \partial_t \phi(t,X_{t^-}) \ dt + \sum_{i:T_i<1} \log \frac{k}{\ell}(T_i,X_{T_i^-})\Big)\\
    &=&  \exp\Big(\int_{(0,1]} \partial_t \phi(t,X_{t^-}) \ dt + \sum_{i:T_i<1} [\phi(T_i,X_{T_i})- \phi(T_i,X_{T_i^-})]  + \sum_{i:T_i<1} c(X_{T_i^-})
\Big)\\
 &=&  \exp \big(\phi(1,X_1) - \phi(0,x)+C(x,X_1) \big) 
\end{eqnarray*}
where $\sum_{i:T_i<1} c(X_{T_i^-} )=c(x)+\cdots+c(X_1-1)=:C(x,X_1)$ only depends on $x$ and $X_1.$
Thus, thanks to Proposition \ref{prop:Q=hP} we deduce that for any $x\in \Z$, $P^x_k \in \Rec (\ell)$ and it follows that $P_k \in
\Rec(\ell)$.
\end{proof}

\begin{remark}\label{rem-rule out}
Defining the reciprocal invariant $\Xi_\ell$ requires  positivity of the intensity $\ell$. 
In particular, any bridge of an NMC process is ruled out since its intensity vanishes at the terminal state. 
This is in some sense a weakness of Definition \ref{def-clasrec} in the present framework of jump processes. It is in contrast with the  diffusion process
setting where reciprocal invariants of bridges are always well defined, see \cite{RT02,RT05}. 
However, we will present in  Theorem  \ref{thm:derivcharacrec}  
a significant improvement of Theorem \ref{thm:clarkpoisson}.
\end{remark}

\begin{example} 
\begin{enumerate}[(a)]
\item
It is immediate to see that for any constant intensities $ \alpha, \beta>0,$ we have $\Rec( \alpha)=\Rec (\beta)$, meaning that any Poisson process with a constant intensity has the same bridges as the standard Poisson process. This was the content of  Example \ref{ex:Poissonbridge}-(a).
 \item 
 Consider two time-homogeneous intensities $\ell(x)$ and $k(x)$. 
 Their associated reciprocal invariants are also time-homogeneous: $\Xi_{\ell}(t,x) = \ell(x+1) - \ell(x)$. 
Then 
$$
\Rec({\ell}) = \Rec(k) \, \Leftrightarrow \,  \forall x \in \Z, \, \ell(x+1) - \ell(x) = k(x+1)-k(x), \;
$$
which means that $\ell - k$ is a 1-periodic function on $\Z$. This implies that the function $k$ is equal to $\ell$ up to some
additive  constant $ \lambda$: $k(x)\equiv \ell(x)+ \lambda$.\\
Therefore, if $ \lambda>0,$ this means that $P_k$ is the law of the sum (or superposition) 
of  $P _{ \ell}$ with an independent Poisson process with intensity $\lambda$.
\item
 Consider  two space-homogeneous deterministic intensities $\ell(t)$ and $k(t)$.
Their associated reciprocal invariants are  space-homogeneous too: $\Xi_{\ell}(t,x) = (\log \ell)'(t)$.
By Theorem \ref{thm:clarkpoisson},  $\Rec(\ell)=\Rec(k)$ if and only if   $k(t)\equiv c \, \ell(t)$ for some  constant $ c>0$.\\
Therefore, if $0< c\le 1,$ $P_k$ is obtained by $c$-thinning  $P_\ell$ as follows. 
Let $X_0 + \sum _{ i} \delta _{ T_i}$ be the point process with law $P_\ell$. 
Then  $P_k$ is the law of  $X_0 + \sum _i \xi_i \delta _{ T_i}$ where $(
\xi_i)_i$ is an iid sequence of Bernoulli($c$) random variables that is independent of $(T_i)_i$. \\
For a  general $c$, decompose $ c=n+r$ with $n\in\N$ its integer part and $0\le r<1$.
Superposing $n$ independent copies of  $P _{ \ell}$ and a $c$-thinning of it gives the process $P_k$.
 \end{enumerate}
\end{example}

\section{Characterization of reciprocal classes by duality formulas} \label{sec:derivativejump}

The main result of the article is Theorem \ref{thm:derivcharacrec} below. We show that each reciprocal class  coincides with the set of random processes 
for which a duality relation holds between some stochastic integral and some derivative operator on the path space. 
This is in the spirit of the results obtained by Thieullen and the fourth author for Brownian diffusion processes. Indeed, it is shown in \cite{RT02} for one-dimensional case and in  \cite{RT05} for the multidimensional case that  any  Brownian diffusion process satisfies  an  integration by parts formula expressed in terms of its reciprocal invariants
that  fully characterizes its reciprocal class.

We introduce a derivative operator by perturbing the time parametrization of the path and prove a duality formula that holds for any counting  process. 
 Then,  we extend these results to bridges of NMC processes and their mixtures, 
 to obtain finally a characterization of the reciprocal class of an NMC intensity at Theorem \ref{thm:derivcharacrec}.

\subsection*{Directional derivative on $\Omega$} 
\label{sec:pertur}

In the framework of Malliavin calculus, Carlen and Pardoux \cite{CP88} introduced 
a directional derivative on the Poisson space by considering infinitesimal changes of the time parametrization (see also Elliott \& Tsoi \cite{ET93}). We follow a similar 
approach in the context of general counting processes.

\subsubsection*{The perturbation operator}

The perturbation operator is defined in terms of a change of time. 

\begin{definition}[The set $\UU$ of perturbation functions]
The set $\UU$ of perturbation functions consists of all ${\mathcal C}^1$-functions  $u:[0,1] \rightarrow \R$ such that
$ u(0)=u(1)=0$. 
\end{definition}

For any  function $u \in \UU$ and $\e > 0$ small enough, we define the  change of time
$\theta^{\e}_u: [0,1]\rightarrow [0,1]$ by 
$$
\theta^{\e}_u(t)= t + \e \, u(t).
$$
The boundedness of the derivative  $\uu$ of $u$ and the  property  $u(0)=u(1)=0$ ensure  that for any $\e$ small enough, $\theta^{\e}_u$ is indeed a change of time with
$\theta^{\e}_u(0)=0$ and $\theta^{\e}_u(1)=1$. \\
\\
The perturbation operator is defined for any path $ \omega\in\Omega$ by
\begin{equation} \label{def:perturbation}
 \Theta^{\e}_u(\omega)= \omega \circ \theta^{\e}_u.
\end{equation}
Note that the  operator $\Theta_u^\e$ keeps the initial and final values of the path unchanged.

\subsubsection*{The derivative operator}

We are now ready to give the definition of a derivative in the direction of the elements of $\UU$, 
in the spirit of the stochastic calculus of variations.

\begin{definition}[The derivative $\DD_u\Phi$]
Let $\Phi $ be a measurable  real function on $\OO$  and $u\in\UU$ a perturbation function. 
We define 
\begin{equation}\label{eq:unitjder}
{\mathcal{D}}_{u} \Phi    := \lim_{\e \rightarrow 0}\frac{1}{\e}\left(\Phi\circ \Theta^{\e}_u - \Phi\right),
\end{equation}
provided that this limit exists.
\end{definition}
Let us remark that we slightly changed  the notations introduced by Carlen and Pardoux in the context of Malliavin calculus:
 we write $\DD_u$ instead of $\DD_{\uu}$ in \cite{CP88}.

\begin{definition}[The set $ \mathcal{S}$ of simple functions]
We say that $\Phi:\OO\to\R$ belongs to the set $ \mathcal{S}$ of simple functions if there exists  $m\ge1$ such that 
$
\Phi =\varphi\big(X_0;T_1,\dots, T_m\big) 
$
for some $\varphi:\Z\times [0,1]^m\to\R$ such that for all $x\in\Z$, the partial functions $\varphi(x;\cdot)$ are  ${\mathcal
C} ^{ \infty}$-differentiable.
\end{definition}
 
 These functions are differentiable on the Poisson space in a natural way, as  proved in  \cite[Thm.\,1.3]{CP88}. 
 
 \begin{lemma}\label{lem:Derivative_Poisson}
 Let  $\Phi \in  \mathcal{S}$ be a simple function.
 It is differentiable in the direction  of any $u\in\UU$ and one has 
\begin{eqnarray}\label{eq-015}
\DD_u\Phi &=& \DD_u \varphi\big(X_0;T_1,\dots, T_m\big) 
	= - \sum_{j=1}^m \partial_{t_j} \varphi(X_0;T_1,\dots, T_m) u(T_j) \\
	&=& -  \int_{[0, 1]} \Big( \sum_{j=1}^m \partial_{t_j} \varphi(X_0;T_1,\dots, T_m)\1_{[0,T_j]}(t)\Big) \uu(t) dt. \nonumber
\end{eqnarray}
 \end{lemma}

\subsection*{Duality formula. Markov counting process}
\label{sec:dualities}

Proposition \ref{thm:IPFPoi+Poi} below states two  duality relations between the derivative operator $\DD$ and some stochastic integrals. 

\begin{proposition}\label{thm:IPFPoi+Poi} 
 The NMC process $\Pl$ satisfies
 the following duality formulas. For all $\Phi\in \mathcal{S}$  and $ u\in \mathcal{U},$
 \begin{eqnarray}  
\El (\mathcal{D}_{u} \Phi)&=&\El \Big( \Phi \int_{[0,1]} \frac{\partial_t(\ell u)}{\ell}\big[dX_t - \ell(t,X_{t^-}) \, dt\big] \Big) 
 \label{eq:IBPFPoisson}\\
\textrm { and } \quad \El \big( \mathcal{D}_{u} \Phi\big)&=&
\El\Big(\Phi\int _{ [0,1]} \big[ \uu(t)+ \Xi_{\ell}(t,X_{t^-})u(t)\big] \,dX_t\Big).
\label{eq:IBPFPoisson_2}
\end{eqnarray}
We do not make the initial distribution $\mu$ precise since it does not play any role.
\end{proposition}

\begin{proof} 

Consider the left hand side of
(\ref{eq:IBPFPoisson})
$$
\El(\mathcal{D}_{u} \Phi)= \El \Big(\lim_{\e \rightarrow 0}\frac{1}{\e}\big(\Phi\circ \Theta^{\e}_u - \Phi \big) \Big) .
$$
Because of the smoothness of $\Phi$ and the boundedness of $\DD_u\Phi$, see \eqref{eq-015}, the expectation $\El (\mathcal{D}_{u} \Phi)$ is well defined and we
can exchange limit  and expectation. 
Therefore 
\begin{equation}\label{eq-021}
\El (\mathcal{D}_{u} \Phi)= \lim_{\e \rightarrow 0} \frac{1}{\e} \Big(\El \big(\Phi\circ  \Theta^{\e}_u\big) - \El \big(\Phi \big) \Big)
= \lim_{\e \rightarrow 0}\El \Big(\Phi \, \frac{G^{\e,u }-1}{\e}\Big) 
\end{equation}
where $G^{\e,u }:=d\Pl \circ (\Theta^{\e}_u)^{-1}/d\Pl.$\\
Now $\Pl \circ (\Theta^{\e}_u)^{-1}$ is a counting process with uniformly bounded intensity 
$\ell(\theta_u^\e,.)\dot \theta_u^\e = \ell(\theta_u^\e,.) (1 + \e \uu)$.
Therefore,  the Girsanov density of $\Pl \circ (\Theta^{\e}_u)^{-1}$ with respect to $\Pl$ verifies:
\begin{eqnarray*}
 G^{\e,u } 
 	& =&  
	 \exp- \Big(  \int _{ [0,1]}\big[ \ell(\theta_u^\e(t),X_{t^-}) (1 + \e \uu(t))- \ell(t,X_{t^-}) \big]\,dt\Big)
	 \prod_{i:T_i<1} \frac{\ell(\theta_u^\e(T_i),X_{T_i^-}) (1 + \e \uu(T_i))}{\ell( T_i, X_{T_i^-})}\\
 &=& \Big(1 - \e  \sum_{i:T_i<1} \big[ \frac{\ell'}{\ell}(T_i,X_{T_i^-})u(T_i)+ \uu(T_i)\big] \Big)
 \prod_{i:T_i<1} \Big(1 + \e  \big[ \frac{\ell'}{\ell}(T_i,X_{T_i^-})u(T_i)+ \uu(T_i)\big]  \Big)+ o(\e) \\
 	&=& 1 + \e \int_{[0,1]} \frac{\partial_t(\ell u)}{\ell} \big[- \ell(t,X_{t^-}) \, dt + dX_t )\big] + o(\e),\quad \Pl\textrm{-a.s.}
	\end{eqnarray*}
 which implies that 
\begin{eqnarray*}
  ({G^{\e,u } -1})/{\e} &=&  \int_{[0,1]}  \frac{\partial_t(\ell u)}{\ell} \big[dX_t - \ell(t,X_{t^-}) \, dt\big]+o(1),\ \Pl\textrm{-a.s.}
\end{eqnarray*}
 which  leads us to \eqref{eq:IBPFPoisson}.
 In particular, under the standard Poisson process $\Poi$,
 \begin{equation} \label{eq:FIP-Poisson}
  \mathbf{E} (\mathcal{D}_{u} \Phi)
 = \mathbf{E} \Big( \Phi \int_{[0,1]} \uu (t) \big( dX_t - dt \big) \Big)
\end{equation}
as it appears in \cite{CP88}. 

To prove the  duality formula \eqref{eq:IBPFPoisson_2} satisfied by $\Pl$, 
we take advantage of  its mutual absolute continuity with the Poisson measure $\Poi$ with the same initial distribution  as $\Pl$.
The density $G_\ell:= d\Pl/d\Poi$, given at \eqref{eq:girsanovunitjump}, is differentiable in the direction of any   $u \in \UU$
and 
\begin{equation}\label{eq-016}
\mathcal{D}_{u} G_\ell = - \, G_\ell \int_{[0,1] } \Xi_{\ell}(t,X_{t^-})u(t)\, dX_t .
\end{equation}
Indeed, one can write $G_{ \ell}$ as a function of the jump times: $G_{ \ell}= \varphi(X_0; T_1,\dots)$ where 
$$ 
\varphi(x;t_1,\dots)=\exp\Big(-\sum _{ i\ge0:t_i <1}\int _{ [t_i,t_{ i+1}\wedge1)}(\ell(t,x+i)-1)\,dt +\sum _{ i\ge1:t_i <1}\ln \ell(t_i, x+i-1)\Big)
$$
with the convention that $t_0=0$. 
Since the Poisson process performs almost surely finitely many jumps, 
we are allowed to invoke the first identity in  \eqref{eq-015} to obtain \eqref{eq-016}.
\\
Then, for any $\Phi \in \mathcal{S}$, 
\begin{eqnarray*}
\El \big( \Phi \int_{[0,1] } \uu(t)  dX_t  \big) 
&=& \EE \big(  G_\ell  \, \Phi \, \int_{[0,1] } \uu(t)  dX_t  \big)\\
&=& \EE \big(  G_\ell  \, \Phi \, \int_{[0,1] } \uu(t)  (dX_t-dt)  \big)\\
&\overset{ \eqref{eq:FIP-Poisson}}{=}& 
\EE \big(  \mathcal{D}_{u} ( G_\ell \, \Phi) \big) \\
&\overset{ \checkmark}{=}& 
\EE \big( G_\ell \,  \mathcal{D}_{u} \Phi \big)
+ \EE \big( \Phi \, \mathcal{D}_{u} G_\ell \big) \\
&\overset{ \eqref{eq-016}}{=}& \El \big(  \mathcal{D}_{u} \Phi \big)
- \El \big( \Phi \int_{[0,1] } \Xi_{\ell}(t,X_{t^-}) u(t)\, dX_t  \big),
\end{eqnarray*}
where we used at the marked equality the rule for the derivative of a product of functions, as proved in \cite[Thm.\,1.4]{CP88} . 
Therefore
$$
\El \big(  \mathcal{D}_{u} \Phi \big)
= \El \Big(\Phi \int_{[0,1] } \big( \uu(t) +  \Xi_{\ell}(t,X_{t^-}) u(t)\big) dX_t  \Big),
$$
which is the desired result.
\end{proof}

\begin{example}\label{ex:decay2}
Let $\Pl$ be a non-homogeneous Poisson process with  exponential intensity $\ell (t) = e^{\lambda t}, \lambda >0$. Then its reciprocal invariant is the constant  $\lambda$ and the duality formula \eqref{eq:IBPFPoisson_2} reduces to
\begin{equation}\label{eq:decaylaw}
\El ( \mathcal{D}_{u} \Phi)=\El^{xy} \Big( \Phi \int_{[0,1]} (\uu(t) + \lambda u(t)) dX_t  \Big) = \El \Big( \Phi \int_{[0,1]} \uu(t)  (dX_t - \lambda X_t  dt)  \Big)
\end{equation}
for any $\Phi \in \mathcal{S}$ and  $u \in \UU$.
\end{example}

\begin{remark} 
\begin{enumerate}[(a)]
\item
The Markov property of the intensity $\ell$ does not play  any role in the previous argumentation.
Identity \eqref{eq:IBPFPoisson} extends to a  large class of predictable regular intensities. 
\item
Take as intensity $\ell$ a constant number $\alpha >0$. Equation \eqref{eq:IBPFPoisson} reduces to
\begin{equation} \label{eq:IBPFPoissonalpha}
\forall \Phi\in \mathcal{S}, \forall u\in \mathcal{U}, 
\quad \EE_\alpha (\mathcal{D}_{u} \Phi)=  \EE_\alpha \big( \Phi \int_{[0,1]} \uu(t) (dX_t - \alpha \,
dt) \big)
\end{equation}
Relaxing the definition of $ \mathcal{D}_u$ for functions $u$ which do not vanish at time $t=1,$  one can obtain a  characterization of  the set $\{\Poi^\mu_\alpha; \mu \in \mathcal{P}(\Z) \}$ of all  Poisson processes with intensity $\alpha$.
\end{enumerate}
\end{remark}

\subsection*{Duality formula. Reciprocal class}\label{sec:dfreciuni}	

Our main result Theorem \ref{thm:derivcharacrec}
states a characterization of  a reciprocal class  in terms of the duality formula \eqref{eq:IBPFPoisson_2}. 
On the way  to its proof, 
we start noting in Lemma \ref{prop:ibpfunit} that the identity \eqref{eq:IBPFPoisson_2}  remains true for any bridge $\Rl ^{ xy}$ associated with the NMC intensity $\ell$.

\begin{lemma}\label{prop:ibpfunit}
Let $\ell$ be an NMC intensity and $x \le y\in \Z.$ Then the duality formula
\begin{equation}\label{eq:ibpfunit2}
\El^{xy} \big( \mathcal{D}_{u} \Phi\big)=
\El^{xy}\Big(\Phi\int _{ [0,1]} \big[ \uu(t)+ \Xi_{\ell}(t,X_{t^-})u(t)\big] \,dX_t\Big)
\end{equation}
holds for all $\Phi \in \mathcal{S}$ and all
$u\in \UU $.
\end{lemma}

\begin{proof}
Take the identity \eqref{eq:IBPFPoisson_2} and apply it with a test function of the form 
$f_0(X_0) f_1(X_1)  \Phi $ where $f_0,f_1:\Z\to \mathbb{R}$ and $\Phi \in \mathcal{S}$. 
Noting that 
$\mathcal{D}_{u}f_0 (X_0)= \mathcal{D}_{u} f_1 (X_1)= 0$,
we deduce that the identity  \eqref{eq:ibpfunit2} is valid for any bridge  $\Rl^{xy}.$ 
\end{proof}

Next result is concerned with the Poissonian case.  It shows that any Poisson bridge is characterized by a
simple duality formula.

\begin{proposition}\label{pr:rchp}
Let $Q \in \PO$ admit the endpoint marginal $Q_{01} = \delta_{(x,y)}$ with $x\le y\in\Z$. The process $Q$ is the Poisson bridge between $x$ and $y$ if and only if
\begin{equation} \label{equ:IBPFpontPoi}
E_Q (\mathcal{D}_{u} \Phi)
=
E_Q\Big(\Phi\int _{ [0,1]} \uu(t)\,dX_t\Big)
\end{equation}
 holds for any $\Phi \in \mathcal{S}$ and any 
$u \in \UU$. 
\end{proposition}

\begin{proof}  
The direct part is a particular case of Lemma \ref{prop:ibpfunit}.\\ 
Let us prove the converse statement by computing the jump intensity of $Q$ by means of a Nelson stochastic derivative.
\\
 Fix $t\in [0,1)$. We are going to apply  \eqref{equ:IBPFpontPoi} with   $\Phi$ an $X _{ [0,t)}$-measurable function  and $u\in \mathcal{U}$ such that 
 $\uu= \frac{1}{\varepsilon}\1_{ [ t,t+ \varepsilon]}- \frac{1}{1-(t+\varepsilon')}\1_{ [t + \varepsilon',
1]}$ where $0 < \varepsilon <\varepsilon'.$  We obtain the following equality
 $$
 E_{Q}\Big(\frac{1}{\varepsilon}\int_{[ t,t+ \varepsilon]} dX_r \mid X _{ [0,t)}\Big)
=  E_{Q}\Big(\frac{1}{1-(t+\varepsilon')}\int_{[ t + \varepsilon', 1]} dX_r \mid X _{ [0,t)}\Big)
 $$
for every  small enough $\varepsilon'.$
Remark that both sides of the equality are constant as functions of $\e$ and $\e'$. In particular, for almost every $t$
 the stochastic derivative 
 $$
 a(t):=\lim _{ \e\to 0}E_{Q}\Big(\frac{1}{\varepsilon}\int_{[ t,t+ \varepsilon]} dX_r \mid X _{ [0,t)}\Big)
 $$
 exists (and is equal to the right hand side). This shows that $Q$
admits the $dt \,Q(d \omega)$-almost everywhere defined process $(t, \omega)\mapsto a(t, \omega)$ as its intensity.
Letting $\e'$ tend to zero gives
\begin{equation}\label{eq-101}
 a(t)= E_{Q}\Big(\frac{1}{1-t}(X_1 - X_t) \mid X _{ [0,t)}\Big)= \frac{y-X_t}{1-t}.
\end{equation}
We recognize the intensity of a Poisson bridge  at time $t$ with final condition $y$.
\end{proof}

By randomizing the endpoint marginal of $Q$ in \eqref{equ:IBPFpontPoi}, we obtain the following 

\begin{corollary} \label{cor:recpoichar}
 If for any $\Phi \in \mathcal{S}$ and any $u \in \UU$ the duality formula
\begin{equation*}
E_Q (\mathcal{D}_{u} \Phi) 
=
E_Q\Big(\Phi\int _{ [0,1]} \uu(t)\,dX_t\Big)
\end{equation*}
 holds under $Q\in \PO$ where $E_Q\big( X_1 - X_0\big) < +\infty$,  then $Q$ belongs to the reciprocal class $\Rec(1)$ of the Poisson process.
\end{corollary}

Next result  emphasizes that the duality formula (\ref{eq:ibpfunit2}) 
characterizes the reciprocal class  of any NMC intensity $\ell$.  A natural idea would be to follow the guideline of the proof of Proposition \ref{pr:rchp}. Unfortunately, this leads to an implicit equation for the intensity, in contrast with the special Poissonian case where the reciprocal invariant $\Xi _{ \ell=1}$ vanishes and gives \eqref{eq-101}. However, a fruitful method consists in relying on the last corollary and the fact that any $Q\in \PO$ is dominated by some Poisson process.

\begin{theorem}\label{thm:derivcharacrec}
Let $Q \in \PO$ be  such that $E_Q\big( X_1 - X_0\big) < +\infty$.
Then $Q$ is in $\Rec(\ell)$ if and only if the duality formula
\begin{equation} \label{eq:IBPFrecipr}
E_Q (\mathcal{D}_{u} \Phi)
=
E_Q\Big(\Phi\int _{ [0,1]} \big[\uu(t)+ \Xi_{\ell}(t,X_{t^-})u(t)\big]\,dX_t\Big)
\end{equation}
holds for any $\Phi \in \mathcal{S}$ and any 
$u \in \UU$.
\end{theorem}

Note that each term in (\ref{eq:IBPFrecipr}) is meaningful since $\Xi_{\ell}$ is bounded and $X_1 - X_0 \in L^1(Q)$. \\

\begin{proof}
The direct statement follows from
Definition \ref{def-clasrec} and Lemma \ref{prop:ibpfunit}.\\ 
Let us prove its converse. We define the event
$ A_n = {\{X_1-X_0=n\}}$	and consider $n$ such that $Q(A_n)>0$. 
We note that \eqref{eq:IBPFrecipr} is satisfied by the measure $Q^{n}:= \frac{\1_{A_n}}{Q(A_n)}Q$ as well. 
By \eqref{eq-016}, $ G_{\ell}:= dP_\ell/d\Poi $ given at \eqref{eq:girsanovunitjump} is
differentiable in any direction $u \in \mathcal{U}$. 
We  define the probability measure $\widetilde{Q}^n$ as
follows
\begin{equation}
d \widetilde{Q}^n := c \,  G^{-1}_{\ell} d  Q^n
\end{equation}  
where $c$ is the normalising constant. 
Since $G^{-1}_{\ell}$ is uniformly bounded from above and below on $A_n$, $\widetilde{Q}^n$ is well defined.  With \eqref{eq-016}, our assumption \eqref{eq:IBPFrecipr} leads us to 
\begin{multline}\label{eq-44}
E_{{Q}^n} \mathcal{D}_u(G_\ell ^{ -1} \Phi)
	= E_{{Q}^n} \Big(G_\ell ^{ -1} \Phi \int_0^1 [\uu_t+ \Xi_\ell(t)u_t]\, dX_t\Big)\\
		=E_{{Q}^n} \Big(G_\ell ^{ -1} \Phi \int_0^1 \uu_t\, dX_t\Big)+E_{{Q}^n} \Big(\mathcal{D}_u(G_\ell ^{ -1}) \Phi \Big).
\end{multline}
Hence,
\begin{eqnarray*}
E_{\widetilde{Q}^n} (\mathcal{D}_u \Phi ) 
	&=&  c \, E_{Q^{n}} \big( G_{\ell}^{ -1} \mathcal{D}_u \Phi \big) 
= -c \,E_{Q^{n}} ( \mathcal{D}_u (G^{-1}_{\ell})   \Phi ) + c\,E_{Q^{n}} 
\mathcal{D}_u(G^{-1}_{\ell} \Phi )\\
	&\overset{ \eqref{eq-44}}{=}&  c \,E_{Q^{n}} \Big(
G^{-1}_{\ell} \Phi \int _{ [0,1]}  \uu(t) dX_t \Big) 
	=   E_{\widetilde{Q}^n} \Big( \Phi \int _{ [0,1]}  \uu(t) dX_t \Big).
\end{eqnarray*}
It follows from Corollary \ref{cor:recpoichar}  that $\widetilde{Q}^n \in \Rec(1)$. By 
Proposition \ref{prop:Q=hP} there exists $h$ such that $ d\widetilde{Q}^n = h(X_0,X_1)\, \Poi $. But this implies that
$$
dQ^{n} = c ^{ -1} \, G_{\ell} \, d \widetilde{Q}^n =
c ^{ -1} \, G_{\ell} \, h(X_0,X_1) \, d \Poi = c ^{ -1} \, h(X_0,X_1) \, d P_{\ell} 
$$
and therefore $Q^{n} \in \Rec( \ell)$. 
By integrating with respect to $n,$ we obtain that $Q \in \Rec(\ell)$ which is the desired result.
\end{proof}

Theorem \ref{thm:derivcharacrec} improves Theorem \ref{thm:clarkpoisson} significantly because
(i) it is not required a priori that the process which stands in the reciprocal class is an NMC process and (ii)
no explicit expression of its intensity of jump is  required.


\begin{thebibliography}{CDPR}

\bibitem[Ber32]{Bern32}
S.~Bernstein.
\newblock Sur les liaisons entre les grandeurs al\'eatoires.
\newblock {\em Verhand. Internat. Math. Kongr. Z{\"u}rich}, Band I, 1932.

\bibitem[CDPR]{CDPR14}
G.~Conforti, P.~Dai~Pra, and S.~Roelly.
\newblock Reciprocal classes of jump processes.
\newblock Preprint.
  \textsf{http://opus.kobv.de/ubp/volltexte/2014/7077/pdf/premath06.pdf}.

\bibitem[Cla91]{Cl91}
J.M.C. Clark.
\newblock A local characterization of reciprocal diffusions.
\newblock {\em Applied Stoch. Analysis}, 5:45--59, 1991.

\bibitem[CP90]{CP88}
E.~Carlen and E.~Pardoux.
\newblock Differential calculus and integration by parts on {P}oisson space.
\newblock In {\em Stochastics, algebra and analysis in classical and quantum
  dynamics}, pages 63--73. Kluwer Acad. Publ., 1990.

\bibitem[Doo57]{Doob57}
J.L. Doob.
\newblock Conditional {B}rownian motion and the boundary limits of harmonic
  functions.
\newblock {\em Bull. Soc. Math. France}, 85:431--458, 1957.

\bibitem[ET93]{ET93}
R.J. Elliott and A.H. Tsoi.
\newblock Integration by parts for {P}oisson processes.
\newblock {\em J. Multivariate Anal.}, 44:179--190, 1993.

\bibitem[Jac75]{Jac75}
J.~Jacod.
\newblock Multivariate point processes: predictable representation,
  {R}adon-{N}ikod\'ym derivatives, representation of martingales.
\newblock {\em Z. Wahrsch. verw. Geb.}, 31:235--253, 1975.

\bibitem[Jam74]{Jam74}
B.~Jamison.
\newblock Reciprocal processes.
\newblock {\em Z. Wahrsch. verw. Geb.}, 30:65--86, 1974.

\bibitem[LRZ]{LRZ12}
C.~L{\'e}onard, S.~R{\oe}lly, and J.-C. Zambrini.
\newblock Reciprocal processes. {A} measure theoretical point of view.
\newblock Preprint, arXiv:1308.0576.

\bibitem[LZ]{LZ14}
C.~L{\'e}onard and J.-C. Zambrini.
\newblock Entropy minimization and calculus of variations of diffusion
  processes.
\newblock Preprint.

\bibitem[Mur12]{Murr12}
R.~Murr.
\newblock {\em Reciprocal classes of {M}arkov processes. {A}n approach with
  duality formulae}.
\newblock PhD thesis, Universit\"at Potsdam,
  http://opus.kobv.de/ubp/volltexte/2012/6301, 2012.

\bibitem[R{\oe}l14]{Roe13}
S.~R{\oe}lly.
\newblock Reciprocal processes. {A} stochastic analysis approach.
\newblock In {\em Modern Stochastics with Applications}, volume~90 of {\em
  Optimization and Its Applications}, pages 53--67. Springer, 2014.

\bibitem[RT04]{RT02}
S.~Roelly and M.~Thieullen.
\newblock A characterization of reciprocal processes via an integration by
  parts formula on the path space.
\newblock {\em Probab. Theory Related Fields}, 123:97--120, 2004.

\bibitem[RT05]{RT05}
S.~Roelly and M.~Thieullen.
\newblock Duality formula for the bridges of a brownian diffusion: Application
  to gradient drifts.
\newblock {\em Stochastic Processes and their Applications}, 115:1677--1700,
  2005.

\bibitem[Sch31]{Sch31}
E.~Schr\"odinger.
\newblock {\"U}ber die {U}mkehrung der {N}aturgesetze.
\newblock {\em Sitzungsberichte Preuss. Akad. Wiss. Berlin. Phys. Math.},
  144:144--153, 1931.

\end{thebibliography}

\end{document}